\documentclass[12pt, oneside]{amsart}

\usepackage{geometry}                
\usepackage{graphicx}
\usepackage[mathscr]{euscript}
\usepackage{enumerate}
\usepackage{bbm}
\usepackage{wasysym}
\usepackage{amsthm,amssymb,amsmath}
\usepackage{subfigure}
\usepackage{esint}
\allowdisplaybreaks[4]

\begin{document}
\title[A Cauchy--Davenport Theorem for locally compact groups]{A Cauchy--Davenport Theorem for\\ locally compact groups}

\author{Yifan Jing}
\address{Mathematical Institute, University of Oxford, UK}
\email{yifan.jing@maths.ox.ac.uk}

\author{Chieu-Minh Tran}
\address{Department of Mathematics, National University of Singapore, Singapore}
\email{trancm6@nus.edu.sg}

\thanks{YJ~was supported by Ben Green’s Simons Investigator Grant, ID:376201.}


\date{} 

\newtheorem{theorem}{Theorem}[section]
\newtheorem{lemma}[theorem]{Lemma}
\newtheorem{fact}[theorem]{Fact}
\newtheorem{proposition}[theorem]{Proposition}
\theoremstyle{definition}
\newtheorem{definition}[theorem]{Definition}
\newtheorem*{thm:associativity}{Theorem \ref{thm:associativity}}
\newtheorem*{thm:associativity2}{Theorem \ref{thm:associativity2}}
\newtheorem*{thm:associativity3}{Theorem \ref{thm:associativity3}}
\def\tri{\,\triangle\,}
\def\P{\mathbb{P}}
\def\E{\mathbb{E}}

\def\D{\mathscr{D}}
\def\HHH{\mathcal{H}}
\def\HH{\mathscr{K}}
\def\tri{\,\triangle\,}
\def\K{\widehat{K}}
\def\e{\mathbbm{1}}
\def\id{\mathrm{id}}
\def\ndim{\mathrm{ndim}}
\def\d{\,\mathrm{d}}
\def\Q{\mathcal{Q}}
\def\p{\mathsf{p_3}}
\def\T{T^{>\delta}}
\def\BM{\mathrm{BM}}
\def\RR{\mathbb{R}}
\def\TT{\mathbb{T}}
\def\ZZ{\mathbb{Z}}
\def\L{L^+}
\def\x{\widetilde{x}}
\def\inu{\nu}
\def\imu{\mu}
\def\BM{\mathrm{BM}}
\def\supp{\mathrm{supp}}
\def\X{X^*}
\def\Y{Y^*}

\newcommand\NN{\mathbb N}
\newcommand{\Case}[2]{\noindent {\bf Case #1:} \emph{#2}}
\newcommand\inner[2]{\langle #1, #2 \rangle}

\newtheorem{claimx}{Claim}
\newenvironment{claim}
  {\pushQED{\qed}\renewcommand{\qedsymbol}{$\scalebox{1.05}{\Bowtie}$}\claimx}
  {\popQED\endclaimx}

\newtheorem{examplex}[theorem]{Example}
\newenvironment{example}
  {\pushQED{\qed}\renewcommand{\qedsymbol}{$\scalebox{1.05}{\Bowtie}$}\examplex}
  {\popQED\endexamplex}

\begin{abstract}
We generalize the Cauchy--Davenport theorem  to locally compact groups. 
\end{abstract}
\maketitle

\section{Introduction}
A fundamental result in additive combinatorics is the Cauchy--Davenport inequality~\cite{cauchy,davenport}: Suppose $X,Y$ are nonempty subsets of $\ZZ/p\ZZ$ for some prime $p$, then 
\[
|X+Y|\geq\min\{|X|+|Y|-1, p\}.
\]

In this paper, we  generalize the above inequality to all locally compact groups:

\begin{theorem}\label{thm: main}
Let $G$ be a locally compact group, $\mu$ a left Haar measure on $G$, $\nu=\mu_{-1}$ the corresponding right Haar measure on $G$, and $\Delta_G:G\to\RR^{>0}$ is the modular map. Suppose $X,Y$ are nonempty compact subsets of $G$ and $XY$ is a subset of a closed set $E$. Then there are $x_0\in X$, $y_0\in Y$, with $\Delta_G(x_0)=\max_{x\in X}\Delta_G(x)=:\alpha$ and $\Delta_G(y_0)=\min_{y\in Y}\Delta_G(y)=:\beta$, 
\begin{equation}\label{eq: main thm 1,1}
\min\Bigg\{\left(\frac{\nu(X)}{\nu(XY)}+\frac{\mu(Y)}{\mu(XY)}\right)\left(1-\frac{\sup_H\mu(H)}{\alpha\nu(X)+\beta^{-1}\mu(Y)}\right), \frac{\mu(G)}{\mu(XY)}\Bigg\}\leq 1,
\end{equation}
where $H$ ranges over proper compact subgroups of $\ker \Delta_G\cap x_0^{-1}XYy_0^{-1}$ with  $$\mu(H)\leq \min\{\beta^{-1}\mu(E),\alpha\nu(E)\}.$$  In particular, when $G$ is unimodular, 
\begin{equation}\label{eq: uni}
 \mu(XY)\geq \min \{\mu(X)+\mu(Y)-\sup_H\mu(H), \mu(G)\},   
\end{equation}
and $H$ ranges over proper compact subgroups of $G$ with $\mu(H)\leq \mu(E).$ 
\end{theorem}

When $G$ is a cyclic group of order $p$, and take $\mu$ the counting measure on $G$, Theorem~\ref{thm: main} recovers the Cauchy--Davenport theorem, as the only proper subgroup of $\ZZ/p\ZZ$ has size one. 

Let us now explain the main inequality~\eqref{eq: main thm 1,1} in Theorem~\ref{thm: main}. When $G$ is \emph{not} unimodular (and hence noncompact, $\mu(G)/\mu(XY)=\infty$), it tells us that either
\[
\frac{\nu(X)}{\nu(XY)}+\frac{\mu(Y)}{\mu(XY)}\leq 1,
\]
which can be understood as $\mu(XY)\geq\mu(X)+\mu(Y)$ in the unimodular settings, 
or there is a compact subgroup $H$ in $\ker\Delta_G \cap x_0^{-1}XYy_0^{-1}$  such that
\[
\nu(H)=\mu(H)\geq (\alpha\nu(X)+\beta^{-1}\mu(Y))\frac{\frac{\nu(X)}{\nu(XY)}+\frac{\mu(Y)}{\mu(XY)}-1}{\frac{\nu(X)}{\nu(XY)}+\frac{\mu(Y)}{\mu(XY)}},
\]
where $\alpha=\Delta_G(x_0)$ and $\beta=\Delta_G(y_0)$. 
Let us justify the usage of $\alpha$ and $\beta$ in the inequality, as they do not appear in the usual unimodular setting where all the elements have $\Delta_G$ value equal to $1$. Given two compact sets $X$ and $Y$, assume $x_0\in X$ and $y_0\in Y$ are elements such that
\[
\Delta_G(x_0)=\max_{x\in X}\Delta_G(x),\quad\text{and}\quad \Delta_G(y_0)=\min_{y\in Y}\Delta_G(y).
\]
Assume $g_1$ and $g_2$ are two arbitrary elements in $G$. Although $G$ is not unimodular, $XY$ and $g_1XYg_2$ may have different (left or right) measures, they are expected to have the same structure. More precisely, the expected expansion rate of $XY$ with respect to $X$ and $Y$, should be the same as the expansion rate of $g_1XYg_2$ with respect to $g_1X$ and $Yg_2$. Note that $\frac{\nu(X)}{\nu(XY)}+\frac{\mu(Y)}{\mu(XY)}$ is invariant under any left translates of $X$ and right translates of $Y$, which means
\[
\frac{\nu(X)}{\nu(XY)}+\frac{\mu(Y)}{\mu(XY)}=\frac{\nu(g_1X)}{\nu(g_1XYg_2)}+\frac{\mu(Yg_2)}{\mu(g_1XYg_2)},
\]
but $\nu(X)+\mu(Y)$ is not invariant under left translates of $X$ or right translates of $Y$. However, $\alpha\nu(X)$ and $\beta^{-1}\mu(Y)$ are invariant under translations. Indeed, as $\Delta_G$ is multiplicative, 
\[
\max_{x\in g_1X}\Delta_G(x)\nu(g_1X)= \Delta_G(x_0)\Delta_G(g_1)\nu(g_1X)=\Delta_G(x_0)\nu(X)=\max_{x\in X}\Delta_G(x)\nu(X),
\]
and same holds for $\beta^{-1}\mu(Y)$.

The Cauchy--Davenport theorem and its generalizations reflect the expansion (or growth) phenomenon in locally compact groups; for example, when $G=\ZZ/p\ZZ$, $X=Y$, and $|X|<p/2$, it implies 
$$|X+X|/|X| \geq 3/2.$$ Prior to our work, Kneser obtained a generalization of the Cauchy--Davenport theorem for locally compact abelian group~\cite{Kneser}, and Kemperman did so for discrete groups~\cite{Kemperman56} and, more generally, for  unimodular locally compact groups~\cite{Kemperman64}; inequality  \eqref{eq: uni} is a restatement of the second result by Kemperman.  The naive generalization involving only a left Haar measure does not work for nonunimodular groups. Indeed,  if $G$ is connected and nonunimodular, one can easily construct nonempty compact $X,Y \subseteq G$ with  
$$
\mu(XY)< \mu(X).$$
However, Kemperman observed the following  intriguing statement involving both the left and right Haar measure in~\cite{Kemperman64}: If $G$ is connected, $X,Y \subseteq G$ are nonempty and compact, then 
\begin{equation}\label{eq: kempermanfornonunimodular}
\min\Bigg\{\frac{\nu(X)}{\nu(XY)}+\frac{\mu(Y)}{\mu(XY)}, \frac{\mu(G)}{\mu(XY)}\Bigg\}\leq 1.
\end{equation}
Surprisingly, inequalities of this form are necessary for the purpose generalizing the Brunn--Minkowski inequality to an arbitrary locally compact group, even if one only cares about unimodular groups; see~\cite{BM} for details. Indeed, if one work on $\mathrm{SL}_2(\RR)$ (which is unimodular), the affine transformation $ax+b$ group (which is \emph{nonunimodular}) appears naturally in its Iwasawa decomposition. 
Thus one might then ask whether there is a common generalization of \eqref{eq: uni} and \eqref{eq: kempermanfornonunimodular} reflecting the expansion phenomenon in an arbitrary locally compact groups. Our main result of the paper is a response to this question.

The proof of Theorem~\ref{thm: main} also provides insights into the structure of the set $XY$. This result can be viewed as a step toward generalizing Kneser's theorem to locally compact groups. For a detailed discussion, see Example~\ref{eg} and Proposition~\ref{prop}.

The paper is organized as follows. In Section 2, we present some basic background on locally compact groups. In Section 3, we prove Theorem~\ref{thm: main}. The main idea of the proof is to study the structure of the ``minimizers": roughly speaking, given sets $X$ and $Y$, we identify subsets $X_0, Y_0 \subseteq XY$ with the largest possible measures, such that $X_0Y_0 \subseteq XY$. In this case, we demonstrate that $X_0 \cap Y_0$ acts as a stabilizer for both $X_0$ and $Y_0$. Many of the ideas we employ trace back to the work of Kemperman~\cite{Kemperman56, Kemperman64}. In Section 4, we present our structural results and pose some open problems.

\section{Preliminaries on locally compact groups}

To make the paper more accessible to readers without a background in locally compact groups, we include here a brief introduction of the basic properties we used about the Haar measures. 
We say that a measure $\mu$ on a $\sigma$-algebra of subsets of $G$ containing all Borel subsets of $G$ is a \emph{left Haar measure} on $G$ if the following conditions hold:
\begin{enumerate}
     \item (left-translation-invariance)  $\mu(X) =\mu(aX)$ for all $a\in G$ and all measurable sets $X \subseteq G$.
    \item (inner and outer regular) When $X$ is open, $\mu(X) =\sup \mu(K)$ with $K$ ranging over compact subsets of $X$.  When $X$ is Borel,  $\mu(X) = \inf \mu(U)$ when $U$ ranging over open subsets of $G$ containing $X$.
    \item (compactly finite) $\mu$ takes finite measure on compact subsets of $G$.
    \item (measurability characterization) If there is an increasing sequence $(K_n)$ of compact subsets of $X$, and a decreasing sequence $(U_n)$ of open subsets of $G$ with $X \subseteq U_n$ for all $n$ such that $\lim_{n \to \infty} \mu(K_n) = \lim_{n \to \infty} \mu(U_n) $, then $X$ is measurable.
\end{enumerate}
The notion of a {\em right Haar measure} $\nu$ is obtained by replacing (1) by right-translation-invariance. Suppose $\mu$ is a left Haar measure on $G$. Let $\nu=\mu_{-1}$, that is for every Borel set $X$, $\nu(X)=\mu(X^{-1})$. It is easy to see that $\nu$ is a right Haar measure.

The following classical result by Haar makes the above notions enduring features of locally compact group; see e.g.~\cite[Section 2.2]{Folland}.
\begin{fact}
Let $G$ be a locally compact group. Up to multiplication by a positive constant, there is a unique left Haar measure on $G$. A similar statement holds for right Haar measure.
\end{fact}

Given a locally compact group $G$, and $\mu$ is a left Haar measure on $G$. For every $x\in G$, recall that 
\begin{align*}
\Delta_G: G&\to\RR^{>0}\\
x&\mapsto \mu^x/\mu
\end{align*}
is the \emph{modular function} of $G$, where $\mu^x$ is a left Haar measure on $G$ defined by $\mu^x(X)=\mu(Xx)$, for every measurable set $X$. When the image of $\Delta_G$ is always $1$, we say $G$ is \emph{unimodular}, which also means that a left Haar measure is also a right Haar measure. In general, $\Delta_G(x)$ takes values in $\RR^{>0}$, where $\RR^{>0}$ is the multiplicative group of positive real number together with the usual Euclidean topology. The next fact records some basic properties of the modular function, see~\cite[Section 2.4]{Folland}.
\begin{fact}
Let $G$ be a locally compact group with a left Haar measure $\mu$ and a right Haar measure $\nu$. 
\begin{enumerate}
    \item Suppose $H$ is a normal closed subgroup of $G$, then $\Delta_H=\Delta_{G}|_{H}$. In particular, if $H=\ker\Delta_G$, then $H$ is unimodular.
    \item The function $\Delta_G: G\to \RR^{>0}$ is a continuous homomorphism. 
    \item For every $x\in G$ and every measurable set $X$, we have $\mu(Xx)=\Delta_G(x)\mu(X)$, and $\nu(xX)=\Delta_G^{-1}(x)\nu(X)$.
    \item If $\nu=\mu_{-1}$, then  $\int_G f\d \mu = \int_G f\Delta_G\d \nu$ for every $f$ is the space $C_c(G)$ of compactly supported continuous function on $G$.  
\end{enumerate}
\end{fact}

\section{Proof of Theorem~\ref{thm: main}}

In this section, we prove our main theorem. 

\begin{proof}[Proof of Theorem~\ref{thm: main}]
Let $\rho \in \RR$ be such that
\begin{equation}\label{eq: first}
\frac{\nu(X)}{\nu(XY)}+\frac{\mu(Y)}{\mu(XY)}=1+\rho.
\end{equation}
We may assume $\rho>0$, as when $\rho\leq0$ the conclusion is immediate.
Now recall that the modular function $\Delta_G$ is continuous, so there is $x_0\in X$ and $y_0\in Y$ such that
\[
\Delta_G(x_0)=\max_{x\in X}\Delta_G(x)\quad\text{and}\quad\Delta_G(y_0)=\min_{y\in Y}\Delta_G(y)
\]
Set $\X=x_0^{-1}X$ and $\Y=Yy_0^{-1}$. Let $G_{\geq1}=\{x\in G: \Delta_G(x)\geq1\}$, and $G_{\leq1}=\{x\in G: \Delta_G(x)\leq1\}$. Then $\X\subseteq G_{\leq1}$, $\Y\subseteq G_{\geq1}$, $\id_G\in \X\cap \Y$. 
By the continuity of $\Delta_G$, both $G_{\geq1}$ and $G_{\leq1}$ are closed. 


Let $\Omega$ be the collection of pairs of sets $(X',Y')$ such that $X'$ is $\nu$-measurable, $Y'$ is $\mu$-measurable,
\[
\nu(X'\setminus (\X\Y\cap G_{\leq1}))=0, \quad \mu(Y'\setminus (\X\Y\cap G_{\geq1}))=0,
\]
and
\[
(\nu\times \mu)\{(x,y): x\in X', y\in Y', xy\notin \X\Y\}=0.
\]

The following claim tells us one can choose a pair of sets from $\Omega$ with the largest possible sum of measures of the two chosen sets. 
\begin{claim}\label{clm:1}
There is $(X_0,Y_0)\in\Omega$, such that for every other $(X',Y')\in \Omega$, either 
\[
\nu(X')+\mu(Y')<\nu(X_0)+\mu(Y_0),
\]
or $\nu(X')+\mu(Y')=\nu(X_0)+\mu(Y_0)$ and $\nu(X')\leq\nu(X_0)$. Moreover, we can arrange to also have $X_0\subseteq G_{\leq1}$, $Y_0\subseteq G_{\geq1}$, $X_0Y_0\subseteq\X\Y$, and $X_0,Y_0$ compact. \medskip

\noindent\emph{Proof of Claim~\ref{clm:1}.}
For any $(X',Y')\in\Omega$, we have
\[
\iint_{G\times G} \e_{X'}(x)\e_{Y'}(y)(1-\e_{\X\Y}(xy))\d\nu(x)\d\mu(y)=0.
\]
Let $\Gamma_\nu$ be the real vector space of equivalent classes of real valued $\nu$-measurable functions on $G$  having a finite $L^1$-norm where we identify two functions if they only differ on a  $\nu$-null set.
Likewise, let $\Gamma_\mu$ be the real linear vector space of all real valued $\mu$-measurable functions on $G$ having a finite $L^1$-norm where we identify two functions if they only differ on a  $\mu$-null set. We equip $\Gamma_\nu$ and $\Gamma_\mu$ with their weak topology associated to the $L^1$-norm; for definitions, see~\cite{weak}. Now we define $K_\nu,K_\mu$ the subsets of $\Gamma_\nu,\Gamma_\mu$ such that 
\[
K_\nu=\{[f]\in \Gamma_\nu: f(x)\in [0,1] \text{ for all }x, \text{ and } f(x)=0\text{ when } x\notin \X\Y\cap G_{\leq1}\},
\]
and
\[
K_\mu=\{[f]\in \Gamma_\mu: f(x)\in [0,1] \text{ for all }x, \text{ and } f(x)=0\text{ when } x\notin \X\Y\cap G_{\geq1}\}.
\]
In the above definitions, we use $[f]$ to denote the equivalent class of the  function $f$ in $\Gamma_\nu$ or $\Gamma_\mu$. Moving forward, we will abuse the notation and write $f$ instead of $[f]$ when talking about elements of $\Gamma_\nu$ or $\Gamma_\mu$.
It is easy to check that $K_\nu$ and $K_\mu$ are closed and sequentially compact subsets (with respect to the weak topology) of $\Gamma_\nu$ and $\Gamma_\mu$ respectively. Indeed, let $\{h_i\}$ be a sequence of functions in $K_\nu$ that is converging weakly to $h$ in $\Gamma_\nu$. Then for any $\e_E$ with $E\subseteq G$ being measurable, as functions in $K_\nu$ are bounded between $0$ and $1$, we have
\[
0\leq \langle h,\e_E\rangle_\nu=\lim_{i\to\infty}\langle h_i,\e_E\rangle_\nu \leq \nu(E),\quad\text{and}\quad \int_{G\setminus X^*Y^*\cap G_{\leq 1}}h\d\nu=0,
\]
and as usual, $\langle f,g\rangle_\nu:=\int fg\d\nu$. 
One can then see that $h\in K_\nu$, as otherwise, choose $E$ to be the set of $x$ with $h(x)\leq -\varepsilon$, or $h(x)\geq 1+\varepsilon$, one can derive a contradiction. This means $K_\nu$ is closed, same argument also works for $K_\mu$. It remains to show that $K_\nu$ is sequentially compact. Firstly, the fact that $\nu(X^*Y^*)<\infty$, implies that $\int |h|\d\nu<\infty$ for every $h\in K_\nu$. Secondly, note that whenever $\{E_i\}$ is a decreasing sequence of sets with empty intersection, 
\[
\lim_{i\to\infty}\langle h,\e_{E_i}\rangle_\nu =0 
\]
holds uniformly for $h\in K_\nu$. Those imply that $K_\nu$ is sequentially compact, and same argument also works for $K_\mu$.

Now we consider the space $\Gamma_\nu\times\Gamma_\mu$ equipped with the product topology. Then  $K_\nu\times K_\mu$ is a closed and sequentially compact  subset of $\Gamma_\nu\times\Gamma_\mu$. Set
\begin{equation}\label{eq: fgh}
\Psi(f,g)=\iint_{G\times G}f(x)g(y)(1-\e_{\X\Y}(xy))\d\nu(x)\d\mu(y).
\end{equation}
Now we will show that $\Psi(f,g)$ is  continuous on $K_\nu\times K_\mu$. 
As $\X\Y$ is measurable,   for each $\varepsilon$, there are finitely many bounded continuous functions $\phi_k$, $\psi_k$ for $1\leq k\leq N$, such that
\[
\iint_{G\times G}\left|(1-\e_{\X\Y}(xy))-\sum_{k=1}^N\phi_k(x)\psi_k(y)\right|\d\nu(x)\d\mu(y)<\varepsilon.
\]
On the other hand, since
\[
(f,g)\mapsto\sum_{k=1}^N\iint_{G\times G}f(x)g(y)\phi_k(x)\psi_k(y)\d\nu(x)\d\mu(y)
\]
is a continuous function on $\Gamma_\nu\times\Gamma_\mu$, hence $\Psi(f,g)$
 is a continuous function on $K_\nu\times K_\mu$. 

Let $\Lambda\subseteq K_\nu\times K_\mu$ be the collections of $(f,g)$ such that $\Psi(f,g)$ is $0$. It is again easy to check that $\Lambda$ is closed and sequentially compact. Moreover, as $(\e_{\X},\e_{\Y})\in \Lambda$, we have  $\Lambda\neq\varnothing$. 

Finally, let us consider the function
\[
\Phi(f,g)=\int_Gf(x)\d\nu(x)+\int_Gg(x)\d\mu(x).
\]
This is a continuous function on $\Gamma_\nu\times\Gamma_\mu$. As $\Lambda$ is sequentially compact, there is a nonempty subset $\Lambda'\subseteq \Lambda$ such that for every $(f',g')\in \Lambda'$,
\[
\Phi(f',g')=\max_{(f,g)\in\Lambda}\Phi(f,g). 
\]
As $\Lambda$ is nonempty, closed, and sequentially compact, there is $(f_0,g_0)\in\Lambda'$ such that $\int f_0\d\nu$ attains the maximum of $\int f\d\nu$ for all $f$ that $(f,g)\in\Lambda'$ for some $g\in\Gamma_\nu$. 

 Let $X_0=\supp(f_0)$ and $Y_0=\supp(g_0)$. One has $(X_0,Y_0)\in\Omega$, and for every $(X',Y')\in\Omega$ we have
\begin{align*}
\nu(X')+\mu(Y')&=\int_G\e_{X'}(x)\d\nu(x)+\int_G\e_{Y'}(x)\d\mu(x)\\
&\leq \int_Gf_0(x)\d\nu(x)+\int_Gg_0(x)\d\mu(x)\leq\nu(X_0)+\mu(Y_0). 
\end{align*}
When the equality holds in the above inequality, we have
\[
\nu(X')=\int_G\e_{X'}(x)\d\nu(x)\leq \int_Gf_0(x)\d\nu(x)\leq\nu(X_0).
\]

As $(X_0,Y_0)\in\Omega$, we have $\nu(X_0\setminus (\X\Y\cap G_{\leq1}))=0$, $\mu(Y_0\setminus (\X\Y\cap G_{\geq1}))=0$, and
\[
(\nu\times \mu)\{(x,y): x\in X_0, y\in Y_0, xy\notin \X\Y\}=0.
\]
Next we are going to modify $X_0$ and $Y_0$ to ensure that $X_0\subseteq G_{\leq1}$, $Y_0\subseteq G_{\geq1}$, and $X_0Y_0\subseteq \X\Y$. 

Let $\nu_{X_0}$ be the measure restricted to $X_0$, that is $\nu_{X_0}(Z)=\nu(X_0\cap Z)$ when $Z$ is measurable. Let $\supp(\nu_{X_0})$ be the support of the measure $\nu_{X_0}$, that is a set of elements $x$ in $G$ such that each open neighborhood $U$ of $x$ satisfies that $\nu_{X_0}( U)>0$, equivalently $\nu_{G}(X_0\cap U)>0$. We similarly define $\mu_{Y_0}$ and $\supp(\mu_{Y_0})$. Clearly 
\[
\nu(X_0)\leq \nu(\supp(\nu_{X_0}))\text{ and }\mu(Y_0)\leq \mu(\supp(\mu_{Y_0})).
\]
As $\X\Y$, $G_{\geq 1}$, and $G_{\leq 1}$ are closed, one can check that $(\supp(\nu_{X_0}),\supp(\mu_{Y_0}))$ is in $\Omega$. Hence by the maximality of $(X_0,Y_0)$, 
\[
\nu(X_0)= \nu(\supp(\nu_{X_0}))\quad\text{and}\quad\mu(Y_0)= \mu(\supp(\mu_{Y_0})). 
\]
By replacing $X_0$, $Y_0$ if necessary, we may assume that 
\begin{equation}\label{eq: supp}
X_0=\supp(\nu_{X_0})\quad\text{and}\quad Y_0=\supp(\mu_{Y_0}).     
\end{equation}
In particular, $X_0$ and $Y_0$ are closed.
Also with the  $(X_0,Y_0)\in \Omega$ knowledge, we get $X_0,Y_0$ are closed and hence compact subsets of $\X\Y$, $X_0\subseteq G_{\leq1}$, and $Y_0\subseteq G_{\geq1}$. Note that we also have $X_0Y_0\subseteq\X\Y$ since the product $xy$ is jointly continuous in $x$ and $y$. 
This proves the claim.
\end{claim}

Now we fix such a pair $(X_0,Y_0)$. Let us remark that the structures of $X_0$ and $Y_0$ might be very different from the structures of the original sets $X$ and $Y$, as $X_0,Y_0$ are obtained purely based on the product set $XY$. 
That means the later proof on the structures of $X_0$ and $Y_0$ provides no structural information on $X$ and $Y$. However,  understanding $X_0$ and $Y_0$ will help us to understand the structure of $\X\Y$ and thus $XY$.

We define $\Omega_0$ to be the collection of pairs of compact sets $(X',Y')$ such that 
\[
X'\subseteq \X\Y\cap G_{\leq1},\quad  Y'\subseteq \X\Y\cap G_{\geq1}, \text{ and } X'Y'\subseteq\X\Y.
\]
Clearly, $\Omega_0\subseteq\Omega$ and by Claim 1, $(X_0,Y_0)\subseteq \Omega_0$.

Let $H=X_0\cap Y_0$, then $H$ is compact, and belongs to $\ker\Delta_G$. 

\begin{claim}\label{clm:str}
$X_0H=X_0$ and $HY_0=Y_0$.
\medskip

\noindent\emph{Proof of Claim~\ref{clm:str}.}
Observe that, for every $(X',Y')\in\Omega_0$, for every $g\in X'\cap Y'$, we have the following property:
\[
(X'\cup X'g, Y'\cap g^{-1}Y')\in\Omega_0,\text{ and }(X'\cap X'g^{-1}, Y'\cup gY')\in\Omega_0.
\]
This is because
\[
(X'\cup X'g)(Y'\cap g^{-1}Y')\subseteq X'Y'\subseteq \X\Y,
\]
and $\Delta_G(g)=1$. Likewise for $(X'\cap X'g^{-1}, Y'\cup gY')$.

Now we fix $h\in H$, and consider pairs of sets
\[
(X_0\cup X_0h, Y_0\cap h^{-1}Y_0),\quad\text{and}\quad (X_0\cap X_0h^{-1}, Y_0\cup hY_0). 
\]
Note that both of the pairs are in $\Omega$. Using the definition of $\Omega$ and the fact that $(X_0\cup X_0h, Y_0\cap h^{-1}Y_0)$ is a pair of sets in $\Omega$, we have that either
\begin{equation}\label{eq: alternative}
   \nu(X_0\cup X_0h)+\mu(Y_0\cap h^{-1}Y_0)<\nu(X_0)+\mu(Y_0),
\end{equation}
or
\begin{equation}\label{eq: alternative2}
    \nu(X_0\cup X_0h)+\mu(Y_0\cap h^{-1}Y_0)=\nu(X_0)+\mu(Y_0),\quad\text{and}\quad \nu(X_0\cup X_0h)\leq \nu(X_0). 
\end{equation}
Observe that 
\[
\nu(X_0\cup X_0h)=\nu(X_0\setminus X_0h)+\nu(X_0h)=\nu(X_0\setminus X_0h)+\nu(X_0)
\]
as well as 
\[
\mu(Y_0\cap h^{-1}Y_0)=\mu(Y_0) - \mu(Y_0\setminus h^{-1}Y_0),
\]
from \eqref{eq: alternative} and \eqref{eq: alternative2} we have either
\[
\nu(X_0\setminus X_0h)<\mu(Y_0\setminus h^{-1}Y_0),
\]or
\[
\nu(X_0\setminus X_0h)=\mu(Y_0\setminus h^{-1}Y_0)\quad\text{and}\quad \nu(X_0\cup X_0h)\leq \nu(X_0). 
\]

Also by Claim~\ref{clm:1}, as $(X_0\cap X_0h^{-1}, Y_0\cup hY_0)\in\Omega$, we have either
\[
\nu(X_0\setminus X_0h)>\mu(Y_0\setminus h^{-1}Y_0),
\]
or
\[\nu(X_0\setminus X_0h)=\mu(Y_0\setminus h^{-1}Y_0)\quad\text{and}\quad \nu(X_0\cap X_0h)\leq \nu(X_0). 
\]
Hence, the only possibility is $\nu(X_0\setminus X_0h^{-1})=\mu(Y_0\setminus h^{-1}Y_0)=0$. 

It remains to show that both $X_0\setminus X_0h^{-1}$ and $Y_0\setminus h^{-1}Y_0$ are empty. Suppose $\widetilde{x}\in X_0\setminus X_0h^{-1}$, and hence $\widetilde{x}h\notin X_0$. As $X_0$ is compact, one can find an open neighborhood $U$ of $\widetilde{x}$ such that $Uh\cap X_0=\varnothing$. This implies $x_0\widetilde{x}\notin\supp(\nu_{X_0})$, contradicts the fact that $X_0=\supp(\nu_{X_0})$. Likewise, $Y_0\setminus h^{-1}Y_0$ is empty.
\end{claim}

Using Claim~\ref{clm:str}, we are going to show that $H$ is in fact a compact subgroup. 

\begin{claim}\label{clm:gp}
$H$ is a compact subgroup.\medskip

\noindent\emph{Proof of Claim~\ref{clm:gp}.} It suffices to show that $H$ is a subgroup. By Claim~\ref{clm:str}, as $H=X_0\cap Y_0$, for every $h_1,h_2\in H$, we have $h_1h_2\in X_0h_2= X_0$, and similarly $h_1h_2\in h_1Y_0= Y_0$. Hence $h_1h_2\in X_0\cap Y_0=H$. 

Let $h$ be in $H$. It is easy to see that $hH\subseteq H$ is compact and closed under multiplication. Consider the collection $\mathscr{C}$ of all nonempty compact subsets of $hH$ which is closed under multiplication. Ordering $\mathscr{C}$ by inclusion, then every chain in $\mathscr{C}$ has a lower bound in $\mathscr{C}$ by compactness. By Zorn's lemma, $hH$ contains a minimal nonempty compact subset $H'$ which is closed under  multiplication. As $H'$ is minimal, for every $h'\in H'$, $h'H'=H'h'=H'$. Thus $H'$ contains the identity and $h'^{-1}$ for every $h'\in H'$, hence it is a group. Since $H'\subseteq hH$, this implies that $\id_G\in hH$, hence $h^{-1}\in H$, which implies that $H$ is a group. 
\end{claim}

Note that as $X_0\subseteq G_{\geq1}$ and $Y_0\subseteq G_{\leq1}$, we have
\begin{align}\label{eq: nu(X)}
    \mu(X_0)=\int_G\e_{X_0}(x)\d\mu=\int_G\Delta_G(x)\e_{X_0}(x)\d\nu\geq\nu(X_0),
\end{align}
and
\begin{align}\label{eq: mu(Y)}
    \nu(Y_0)=\int_G\e_{Y_0}(x)\d\nu=\int_G\Delta_G^{-1}(x)\e_{Y_0}(x)\d\mu\geq\mu(Y_0). 
\end{align}

The next claim shows that $H$ is large.

\begin{claim}\label{clm:size}
Let 
\[
\kappa=\frac{(\nu(\X)+\mu(\Y))\nu(\X\Y)\mu(\X\Y)}{\nu(\X)\mu(\X\Y)+\mu(\Y)\nu(\X\Y)}.
\]
Then $\mu(H)\geq \rho\kappa$.
\medskip

\noindent\emph{Proof of Claim~\ref{clm:size}.} 
Since $X_0\cup Y_0\subseteq X_0Y_0$, by the inclusion–exclusion principle, as well as \eqref{eq: nu(X)} and \eqref{eq: mu(Y)},
\begin{equation}\label{eq: small mu}
\mu(H)\geq \mu(X_0)+\mu(Y_0)-\mu(X_0Y_0)\geq \nu(X_0)+\mu(Y_0)-\mu(X_0Y_0),
\end{equation}
and 
\begin{equation}\label{eq: small nu}
\nu(H)\geq \nu(X_0)+\nu(Y_0)-\nu(X_0Y_0)\geq \nu(X_0)+\mu(Y_0)-\nu(X_0Y_0).
\end{equation}
Since $H$ is a group, by the choice of $\nu$, we have $\mu(H)=\nu(H^{-1})=\nu(H)$. As
\[
\frac{1}{\kappa}=\frac{\nu(\X)}{\nu(\X)+\mu(\Y)}\frac{1}{\nu(\X\Y)}+\frac{\mu(\Y)}{\nu(\X)+\mu(\Y)}\frac{1}{\mu(\X\Y)},
\]
we have $\min\{\nu(\X\Y),\mu(\X\Y)\}\leq\kappa\leq \max\{\nu(\X\Y),\mu(\X\Y)\}$. This in particularly implies that \[
\kappa\geq\min\{\nu(X_0Y_0),\mu(X_0Y_0)\}.
\]

Therefore by Claim~\ref{clm:1} and \eqref{eq: small mu}, \eqref{eq: small nu},
\begin{align*}
\frac{\mu(H)}{\kappa}&\geq \frac{\nu(X_0)}{\kappa}+\frac{\mu(Y_0)}{\kappa}-1\geq \frac{\nu(X^*)}{\kappa}+\frac{\mu(Y^*)}{\kappa}-1\\
&=\frac{\nu(\X)}{\nu(\X\Y)}+\frac{\mu(\Y)}{\mu(\X\Y)}-1=\rho,
\end{align*}
and this proves the claim.
\end{claim}
Let us first verify that Claim~\ref{clm:size} is enough to derive \eqref{eq: main thm 1,1}. By elementary computation, together with the definition of $\rho$, \eqref{eq: main thm 1,1} is equivalence to
\begin{align*}
\mu(H)&\geq\frac{\rho(\alpha\nu(X)+\beta^{-1}\mu(Y))}{\rho+1}=\frac{\rho(\nu(X^*)+\mu(Y^*))}{\frac{\nu(X)}{\nu(XY)}+\frac{\mu(Y)}{\mu(XY)}}\\
&=\frac{\rho(\nu(X^*)+\mu(Y^*))}{\frac{\nu(X^*)}{\nu(X^*Y^*)}+\frac{\mu(Y^*)}{\mu(X^*Y^*)}}=\rho\kappa,
\end{align*}
which is proven in the claim. 

Also, since $H=X_0\cap Y_0\subseteq X^* Y^*\subseteq x_0^{-1}XYy_0^{-1}$, we have $y_0^{-1}Hy_0\subseteq y_0^{-1}x_0^{-1}XY$, and $x_0Hx_0^{-1}\subseteq XYy_0^{-1}x_0^{-1}$. Thus
\[
\mu(H)\leq\min\{\Delta_G(y_0)^{-1}\mu(XY), \Delta_G(x_0)\nu(XY)\}.
\] 
As $XY\subseteq E$, the result follows. 

Let us also remark that $H$ is proper as it is compact and inside $\ker(\Delta_G)$. 
\end{proof}

\section{Concluding Remarks}

In~\cite{Kneser}, Kneser proved a stronger result than the abelian version of Theorem~\ref{thm: main}: If $G$ is a locally compact abelian group equipped with a Haar measure $\mu$, and $X,Y$ are nonempty compact subsets of $G$, then there is an open subgroup $H$ such that
\begin{enumerate}[\rm (i)]
    \item $\mu(XY)\geq\mu(X)+\mu(Y)-\mu(H)$.
    \item $XY=XYH$.
\end{enumerate}
It was shown by Olson~\cite{olson} that the statement (ii) in the Kneser theorem cannot be extended to nonabelian groups, even if we replace it by a weaker condition that either $XY=XYH$, or $XY=XHY$, or $XY=HXY$. 

In a recent breakthrough, DeVos~\cite{Devos} characterized finite sets $X,Y$ in a possibly nonabelian group with $|XY|<|X|+|Y|$. As a corollary of his main result, he obtained a generalization of the Kneser theorem to \emph{discrete groups}, with a weakening version of statement (ii):
\[
\text{ for every }g\in XY,\text{ there is }z\in G\text{ such that }g(zHz^{-1})\subseteq XY.
\]

It would be very interesting if such a result can be obtained for general locally compact groups. However, the following example given by Kemperman~\cite{Kemperman64} suggested the problem will be difficult:

\begin{example}\label{eg}
Let $H\leq\ker\Delta_G$ be a compact group. Let $X$ be an arbitrary compact subset of $H$, and $Y=H\cup Wx$ where $W$ is an arbitrary compact set with $H\cap Wx=\varnothing$. Assume $\Delta_G(x)$ is sufficiently small, then we always have
\[
\frac{\nu(X)}{\nu(XY)}+\frac{\mu(Y)}{\mu(XY)}=\frac{\nu(X)}{\nu(H)+\nu(XW)}+\frac{\mu(H)+\mu(W)\Delta_G(x)}{\mu(H)+\mu(XW)\Delta_G(x)}>1.
\]
As $X,W$ are chosen arbitrarily, the structure of $XW$ is hard to control.
\end{example}

On the other hand, from our proof of Theorem~\ref{thm: main}, we have a structural control on the ``majority'' of $XY$: 
\begin{proposition}\label{prop}
Let $G$ be a locally compact group, $\mu$ a left Haar measure on $G$, $\nu=\mu_{-1}$ the corresponding right Haar measure on $G$, and $\Delta_G:G\to\RR^{>0}$ is the modular map. Suppose $X,Y$ are nonempty compact subsets of $G$, and set $\alpha=\sup_{x\in X}\Delta_G(x)$, $\beta=\inf_{y\in Y}\Delta_G(y)$. Then
\begin{enumerate}[\rm (i)]
    \item there is an open subgroup $H$ of $\ker\Delta_G$ such that \[
\min\Bigg\{\left(\frac{\nu(X)}{\nu(XY)}+\frac{\mu(Y)}{\mu(XY)}\right)\left(1-\frac{\mu(H)}{\alpha\nu(X)+\beta^{-1}\mu(Y)}\right), \frac{\mu(G)}{\mu(XY)}\Bigg\}\leq 1,
\]
\item there is $D\subseteq XY$ with 
\begin{align*}
&\,\min\{\beta^{-1}\mu(D),\alpha\nu(D)\}\\\leq&\, \mu(H)-\left(\frac{\nu(X)}{\nu(XY)}+\frac{\mu(Y)}{\mu(XY)}-1\right)\frac{(\alpha\nu(X)+\beta^{-1}\mu(Y))\nu(XY)\mu(XY)}{\nu(X)\mu(XY)+\mu(Y)\nu(XY)}
\end{align*}
such that the following hold:
for every $g\in XY\setminus D$, there exists $z\in G$ such that $g(zHz^{-1})\subseteq XY$ with the above $H$.
\end{enumerate}
\end{proposition}
\begin{proof}
Let $X^*,Y^*,X_0,Y_0,H$ be the sets defined in the proof of Theorem~\ref{thm: main}. Set $D=XY\setminus x_0X_0Y_0y_0$. Let $a\in x_0X_0$, $b\in Y_0y_0$, and $g=ab$. Note that we have  $H\subseteq X_0Y_0\subseteq x_0^{-1} XY y_0^{-1}$, where $x_0\in X_0$ and $y_0\in Y_0$.  Then \[
g^{-1}XY=b^{-1}a^{-1}XY= b^{-1}a^{-1}x_0X^*Y^*y_0\supseteq b^{-1}a^{-1}x_0X_0Y_0y_0.
\]
By Claim~\ref{clm:str}, $Hby_0^{-1}\subseteq Y_0$. As $\id_G\in a^{-1}x_0X_0$, $b^{-1}\in b^{-1}a^{-1}x_0X_0$. Thus $b^{-1}Hb\subseteq b^{-1}a^{-1}x_0X_0Y_0y_0\subseteq g^{-1}XY$. 

It remains to show that $ D$ has ``small measure''. Without loss of generality, we assume $\mu(\X\Y)=\min\{\mu(\X\Y),\nu(\X\Y)\}$. Let $\kappa$ be as in Claim~\ref{clm:size}. Then $\mu(\X\Y)-\kappa\leq 0$. By \eqref{eq: small mu} and \eqref{eq: small nu}, we have
\begin{align*}
  \frac{\mu(X^*Y^*)-\mu(x_0^{-1}Dy_0^{-1})}{\kappa}=\frac{\mu(X_0Y_0)}{\kappa}&\geq \frac{\nu(\X)}{\kappa}+\frac{\mu(\Y)}{\kappa}-\frac{\mu(H)}{\kappa}
  \\&\geq \frac{\nu(\X)}{\nu(\X\Y)}+\frac{\mu(\Y)}{\mu(\X\Y)}-\frac{\mu(H)}{\kappa}.
\end{align*}
This implies
\[
\beta^{-1}\mu(D)=\mu(x_0^{-1}Dy_0^{-1})\leq \mu(H)-\rho\kappa+\mu(\X\Y)-\kappa=\mu(H)-\rho\kappa
\]
with $\rho$ defined in \eqref{eq: first}. This proves statement (ii).
\end{proof}

In Example~\ref{eg}, while $\mu(XWx)$ is small, $\nu(XWx)=\nu(XW)$ can be very large. It suggests that one of the $\mu(D)$ and $\nu( D)$ can be large, hence to get an upper bound on $\min\{\mu( D),\nu( D)\}$ (as in Proposition~\ref{prop}) might be the best thing we can get. Nevertheless, we believe the bound in Proposition~\ref{prop} is not sharp. In fact, we conjecture that when $G$ is unimodular, $D=\varnothing$.

Another direction is to consider the inverse problems. The Vosper theorem~\cite{Vosper} and the Freiman $2.4k$ theorem~\cite{freiman} characterize subsets of $\ZZ/p\ZZ$ where the equality in the Cauchy--Davenport theorem happens or nearly happens. For unimodular groups, the corresponding questions of characterizing
\[
\mu(XY)=\mu(X)+\mu(Y),\quad\text{or} \quad \mu(XY)\leq\mu(X)+\mu(Y)+\delta,
\]
were asked by Griesmer~\cite{Griesmer}, by Kemperman~\cite{Kemperman64}, and by Tao~\cite{Tao}. The answer for connected locally compact groups were recently obtained by the authors~\cite{JT21} for compact groups and by An, Zhang, and the authors~\cite{AJTZ} for noncompact groups.

One can also ask the similar questions when $G$ contains  subgroups of finite positive measure, and the equality in Theorem~\ref{thm: main} nearly happens: suppose $G$ is unimodular, $H\leq G$,  when will we have
\[
\mu(XY)\leq \mu(X)+\mu(Y)-\mu(H)+\delta?
\]
If $G$ is not unimodular, when will we have
\[
\frac{\nu(X)}{\nu(XY)}+\frac{\mu(Y)}{\mu(XY)}-\frac{\mu(H)}{\min\{\mu(XY),\nu(XY)\}}\geq 1-\delta?
\]
For nonunimodular $G$, the special case when $\mu(XY)=\mu(Y)$ is characterized by Macbeath~\cite{macbeath}.

\section*{Acknowledgements}
The authors thank Shukun Wu for reading the first draft of the manuscript and making many useful comments, and the anonymous referee for their many helpful corrections and suggestions.

\bibliographystyle{amsplain}
\bibliography{reference}

\end{document}